\documentclass{amsart}
\usepackage[utf8]{inputenc}
\usepackage{amsmath}
\usepackage{amssymb}
\usepackage{graphicx}
\usepackage{centernot}
\usepackage{mathtools}
\usepackage{mathrsfs}
\usepackage{enumitem}
\usepackage{dsfont}
\usepackage{verbatim}
\usepackage{cleveref}
\usepackage{graphicx}
\usepackage[all]{xy}

\newtheorem{thrm}{Theorem}[section]
\newtheorem{lem}[thrm]{Lemma}
\newtheorem{prop}[thrm]{Proposition}
\newtheorem{cor}[thrm]{Corollary}
\theoremstyle{definition}

\newtheorem{example}[thrm]{Example}
\newtheorem{remark}[thrm]{Remark}

\newcommand{\R}{\mathbb R} %real numbers
\newcommand{\N}{\mathbb N} %natural numbers
\newcommand{\abs}[1]{\left\lvert#1\right\rvert} %absolute value
\newcommand{\norm}[1]{\left\lVert#1\right\rVert} %norm
\newcommand{\iprod}[1]{\left\langle#1\right\rangle} %inner product
\DeclareMathOperator{\closure}{cl}

\numberwithin{equation}{section}

\author[A. W. Guti{é}rrez]{Armando W. Guti{é}rrez}
\address{
Department of Mathematics and Systems Analysis\\ 
Aalto University\\ 
Otakaari 1 Espoo, Finland}
\email{wladimir.gutierrez@aalto.fi}
\thanks{This work was supported by the Academy of Finland, Grant No. 288318.}

\begin{document}

\title[metric functionals on infinite-dimensional $\ell_{p}$ spaces]
{On the metric compactification of infinite-dimensional $\ell_{p}$ spaces}

\begin{abstract}
	The notion of metric compactification
	was introduced by Gromov and later rediscovered
	by Rieffel; and has been mainly studied on proper
	geodesic metric spaces.  
	We present here a generalization of the
	metric compactification that can be
	applied to infinite-dimensional Banach spaces. 
	Thereafter we give a complete description of the
	metric compactification of infinite-dimensional 
	$\ell_{p}$ spaces for all $1\leq p < \infty$. 
	We also give a full characterization of the metric
	compactification of infinite-dimensional Hilbert spaces. 
\end{abstract}
\maketitle

\section{Introduction}
In consonance with topology and potential theory,
Gromov introduced \cite{Gromov1981,Ballmann_Gromov_Schroeder1985}
a method to attach an ideal boundary $X_{\infty}$
at infinity of a metric space $X$. 
The method consists in mapping $X$ into 
the set of real-valued continuous functions $C(X)$ 
equipped with the topology of uniform convergence on 
bounded subsets. If the metric space $X$ is proper and geodesic, then
Gromov's bordification $X\sqcup X_{\infty}$
becomes a compact topological space which contains $X$ 
as a dense open subset \cite{Ballman1995,Bridson_Haefliger1999}. 
Later, Rieffel \cite{Rieffel2002} obtained the compact space
$X\sqcup X_{\infty}$ as the maximal ideal space
of a unital commutative $C^{*}$-algebra, and termed it
the metric compactification. Although it
is perhaps more often known as the horofunction compactification.
For finite-dimensional Banach spaces, this compactification
has been studied in 
\cite{Friedland_Freitas2004, Karlsson_Metz_Noskov2006, Walsh2007, Ji_schilling2016, Gutierrez2017}.

Compactness is a fundamental tool in mathematics.
Infinite-dimensional Banach spaces
are not locally compact so Gromov's procedure
does not yield a compactification, but only a bordification.
However, there is an alternative method to
compactify an infinite-dimensional Banach space $X$.
By taking instead the topology of pointwise
convergence on $C(X)$ and following 
\cite{Gaubert_Vigeral2012, Maher_Tiozzo2014, Gouzel_Karlsson2015} 
we obtain a metric compactification of $X$ in a weak sense, 
see \Cref{concepts}.
While studying certain metric
geometries of infinite-dimensional convex cones,
Walsh presented in \cite{Walsh2016} a description of 
a certain subset of the metric compactification of
infinite-dimensional Banach spaces, namely the set of 
Busemann points. A Busemann point is an element
of the metric compactification obtained as a limit 
of some almost-geodesic net, a concept that was first introduced 
by Rieffel and then slightly modified by Walsh.

We emphasize that the techniques we use 
in this paper are significantly different from those used by Walsh. 
Moreover, we provide explicit formulas for 
all the elements of the metric compactification of some
classical infinite-dimensional Banach spaces.

Let $J$ be any nonempty index set. We denote by 
$x=(x(j))_{j\in J}$ any element of the real vector space 
$\R^{J}$. For every $1\leq p < \infty$ we consider the
$p$-norm $\norm{\cdot}_p$ on $\R^{J}$ defined by  
$\norm{x}_{p}:=\big( \sum_{j\in J}\abs{x(j)}^{p} \big)^{1/p}$,
where the sum is given by
\begin{equation}\label{sumdef}
	\sum_{j\in J}\abs{x(j)}^{p}=
		\sup \bigg\{ \sum_{j\in F}\abs{x(j)}^{p}\;\bigg\vert\; F \text{ a finite subset of } J \bigg\},
\end{equation} 
for all $x\in\R^{J}$.
We denote by $\ell_p(J)$ the space
of all $x=(x(j))_{j\in J}$ in $\R^{J}$ such that 
$\norm{x}_{p}$ is finite. That is, 
$\ell_{p}(J)=\big\{x\in\R^{J} \;\big\vert\; \norm{x}_{p}<\infty \big\}$.
The space $\ell_{\infty}(J)$ consists
of all $x=(x(j))_{j\in J}$ in $\R^{J}$ such that
$\sup_{j\in J}\abs{x(j)}$ is finite.

In \cite{Gutierrez2017} the author gives a complete
description of the metric compactification of
$\ell_p(\{1,...,N\})$ for all $1\leq p \leq \infty$ and 
$N\in\N$. 
In this paper we study the metric 
compactification of $\ell_p(J)$ for all $1\leq p < \infty$
and $J$ any countably infinite or uncountable index set.

The paper is organized as follows. In \Cref{concepts} we
introduce the notion of metric compactification 
of infinite-dimensional Banach spaces.
In \Cref{horocompl1} we give a complete
description of the metric compactification
of the infinite-dimensional Banach space $\ell_{1}(J)$. 
In \Cref{horoHilbert} we give a
complete description of the metric
compactification of infinite-dimensional
real Hilbert spaces. In \Cref{horolp} 
we give a complete description
of the metric compactification
of the infinite-dimensional Banach space
$\ell_{p}(J)$ for all $1<p<\infty$.

Recent works
have shown that the metric compactification 
provides a powerful modern tool 
for the study of deterministic and random 
dynamics of nonexpansive mappings
\cite{Karlsson_Ledrappier2011, Gaubert_Vigeral2012, Karlsson2014, Gouzel_Karlsson2015}. 
The purpose of this paper is to present explicit formulas 
for all the elements
of the metric compactification of infinite-dimensional 
$\ell_{p}$ spaces.

\section{Preliminaries}\label{concepts}
\subsection{The metric compactification}
Let $(X,d)$ be a metric space. Fix an arbitrary \emph{base point} 
$b$ in $X$. For each $y\in X$ we will consider the element
$h_{b,y}$ of $\R^{X}$ defined by
\begin{equation}\label{emb}
	h_{b,y}(x):=d(x,y)-d(b,y), \text{ for all } x\in X. 
\end{equation}

For every $y\in X$ the 
function $h_{b,y}$ is bounded from below by $-d(b,y)$. Moreover, 
$\{ h_{b,y} \mid y\in X\}$ is a family of $1$-Lipschitz
functions with respect to the metric $d$. Indeed, by the 
triangle inequality we have
\begin{equation*}
	\begin{aligned}
	\abs{h_{b,y}(x)-h_{b,y}(z)} &=\abs{d(x,y)-d(b,y)-d(z,y)+d(b,y)} \\
			&=\abs{d(x,y)-d(z,y)} \\
			&\leq d(x,z),
	\end{aligned}		
\end{equation*}
for all $x,z\in X$. Furthermore, by taking $z=b$ we obtain 
$\abs{h_{b,y}(x)}\leq d(x,b)$ for all $x\in X$. Hence
\begin{equation}\label{subsetcomp}
		\{ h_{b,y} \,\mid\, y\in X \} \subset \prod_{x\in X}[-d(x,b),d(x,b)].
\end{equation}

By Tychonoff's theorem the set on the right-hand side of 
(\ref{subsetcomp}) is compact in the product topology. 
Therefore the set $\{ h_{b,y} \,\vert\, y\in X \}$
has compact closure in this topology, which is equivalent 
to the topology of pointwise convergence. 
Moreover, for any two different base points $b,b'\in X$ the
equality
\begin{equation}\label{homeo}
	h_{b,y}(x)-h_{b,y}(b')=h_{b',y}(x)
\end{equation}
holds for all $x\in X$. Then (\ref{homeo}) 
induces a homeomorphim between the closures 
$\closure(\{ h_{b,y} \,\vert\, y\in X\})$
and $\closure(\{ h_{b',y} \,\vert\, y\in X)$. 
We will write $h_{y}$ instead of $h_{b,y}$, and so we will say that 
\begin{equation}\label{metriccompact}
	\overline{X}^{h}:=\closure(\{ h_{y} \,\vert\, y\in X \})
\end{equation} 
is the \emph{metric compactification} of $(X,d)$. 
As in \cite{Gouzel_Karlsson2015}, we will call all the 
elements of $\overline{X}^{h}$ \emph{metric functionals} 
on $X$. 
By following \cite{Maher_Tiozzo2014}, 
it will be convenient to consider the partition 
$\overline{X}^{h}=\overline{X}^{h,F}\sqcup\overline{X}^{h,\infty}$,
where
\begin{equation*}
\begin{aligned}
	\overline{X}^{h,F} &:=\big\{h\in \overline{X}^{h} \,\big\vert\, \inf_{x\in X} h(x) > - \infty \big\},\\
	\overline{X}^{h,\infty} &:=\big\{h\in \overline{X}^{h} \,\big\vert\, \inf_{x\in X} h(x) = - \infty \big\}.
\end{aligned}
\end{equation*}

The elements of $\overline{X}^{h,F}$ will be called \emph{finite metric functionals}
while the elements of $\overline{X}^{h,\infty}$ will be called \emph{metric functionals at infinity}.
It is clear that $\overline{X}^{h,F}$
contains the set $\{ h_{y} \,\vert\, y\in X \}$ of \emph{internal metric functionals}.

\begin{remark}
Recall that a metric space $(X,d)$ is geodesic if every
pair of points $x,y\in X$ can be joined by a path
$\gamma:[0,d(x,y)]\to X$ such that 
$\gamma(0)=x$, $\gamma(d(x,y))=y$ and
$d(\gamma(s),\gamma(t))=\abs{s-t}$
for all $s,t$. We say that $X$
is proper if every closed ball
of $X$ is compact. 
If $(X,d)$ is a proper geodesic metric space
then the set $\overline{X}^{h}$ in (\ref{metriccompact}) 
coincides with the usual Gromov's
metric compactification. Indeed, 
since $X$ is proper and $\{ h_{y}\,\vert\, y\in X \}$ is a family of
$1$-Lipschitz functions, it follows by the Arzel\`a-Ascoli theorem
that the topology of pointwise and the topology of uniform convergence
on bounded subsets will produce the same compact closure (\ref{metriccompact}).
Moreover, since $X$ is geodesic, the set of internal metric functionals 
$\{ h_{y}\,\vert\, y\in X \}$
can be identified with $X$ 
which becomes a dense open set in $\overline{X}^{h}$. The set
\begin{align*}
	\partial_{h} X:=\overline{X}^{h}\setminus X
\end{align*}
is called the \emph{metric boundary} (or \emph{the horofunction boundary})
of $X$.
For general metric spaces, however, 
although the space $\overline{X}^{h}$ is always compact
with respect to 
the pointwise topology, the injection $y\mapsto h_{y}$ 
may not be an embedding from $X$ to $\overline{X}^{h}$,
see \cite{Gaubert_Vigeral2012} for more details.
\end{remark}

\begin{remark}
If $X$ is a Banach space with metric induced by a 
norm $\norm{\cdot}$, we choose the base point 
$b=0\in X$ so (\ref{emb}) 
becomes 
\begin{equation*}
	h_{y}(x)=\norm{x-y}-\norm{y}.
\end{equation*} 
If $X$ is a finite-dimensional Banach space, 
then every metric functional 
$h\in\overline{X}^{h}$
can be written as $h=\lim_{n\to\infty} h_{y_n}$, for some 
sequence $\{y_n\}_{n\in\N}$ in $X$. 
For an infinite-dimensional Banach space $X$, 
the compact space $\overline{X}^{h}$ may not be metrizable.
However, for every metric functional $h\in\overline{X}^{h}$ 
there will always exist a net $\{y_\alpha\}_{\alpha}$ in $X$ such that 
$h_{y_\alpha}\underset{\alpha}\longrightarrow h$ pointwise on $X$.
For details about convergence of nets we refer 
to \cite{Kelley1975}. 
\end{remark}

The following is a characterization of the finite
metric functionals on Banach spaces.
\begin{lem}
If $\{y_{\alpha}\}_{\alpha}$ is a bounded net in a
Banach space $(X,\norm{\cdot})$ such that 
$h_{y_{\alpha}}\underset{\alpha}\longrightarrow h$
pointwise on $X$, then $h\in \overline{X}^{h,F}$.
\end{lem}
\begin{proof}
Indeed, by passing to a subnet
we may assume that $\norm{y_{\alpha}}\underset{\alpha}\longrightarrow c$.
Also, for every $x\in X$ and for every $\alpha$ we have
\begin{equation*}
	h_{y_{\alpha}}(x)=\norm{x-y_{\alpha}}-\norm{y_{\alpha}} \geq -\norm{y_{\alpha}}. 
\end{equation*}
Then $h(x)\geq -c$ for all $x\in X$, so the claim follows. 
\end{proof}

It is important to notice that if $X$ is a finite-dimensional
Banach space then $\overline{X}^{h,F}=\{h_{y} \,\vert\, y\in X\}$. 
For infinite-dimensional Banach spaces, 
the set $\overline{X}^{h,F}$ of finite metric functionals
may be larger than the set $\{h_{y} \,\vert\, y\in X\}$ 
of internal metric functionals. 

\subsection{Some facts in Banach space theory}
For convenience of the reader we establish some
notations and recall some important facts in
Banach space theory that will be useful in the following sections.
Let $(X,\norm{\cdot})$ be a Banach space with dual space 
$(X^{*},\norm{\cdot}_{*})$. Let $\{x_{\alpha}\}_{\alpha}$ be
any net in $X$ and let $x$ be any vector in $X$. We say
that $x_{\alpha}$ converges strongly to $x$,
and denote by $x_{\alpha} \underset{\alpha}\longrightarrow x$, 
if $\lim_{\alpha}\norm{x_{\alpha}-x} = 0$.
We say that $x_{\alpha}$ converges weakly to $x$, and 
denote by $x_{\alpha} \overset{w}{\underset{\alpha}\longrightarrow} x$, 
if $\lim_{\alpha}\nu(x_{\alpha}) = \nu(x)$ for
all $\nu\in X^{*}$. Let $\{\mu_{\alpha}\}_{\alpha}$ be
any net in $X^{*}$ and let $\mu$ be any vector in $X^{*}$. We say
that $\mu_{\alpha}$ converges to $\mu$ in the weak-star topology, 
and denote by $\mu_{\alpha} \overset{w^{*}}{\underset{\alpha}\longrightarrow} \mu$,
if $\lim_{\alpha}\mu_{\alpha}(x) = \mu(x)$
for all $x\in X$. 

Throughout we will make use of Alaoglu's theorem 
\cite[Prop.~6.13]{Khamsi_Kirk2001}, which
states that the closed unit ball of the
dual space $X^{*}$ is compact in the weak-star topology.
Likewise, it is well known that a
Banach space $X$ is reflexive if and only if
its closed unit ball is compact in the weak topology, see 
\cite[Prop.~6.14]{Khamsi_Kirk2001}. 
In particular, every bounded net in a reflexive
Banach space has a weakly convergent subnet. 

In the following sections we will denote by $c_{0}(J)$ the space
of all $x\in\ell_{\infty}(J)$ such that the set
$\big\{j\in J \,\big\vert\, \abs{x(j)}\geq\epsilon \big\}$ is finite
for all $\epsilon >0$. The space $c_{0}(J)$ is a Banach space
with respect to the norm inherited from $\ell_{\infty}(J)$.
We denote by $c_{00}(J)$ the space
of all $x\in\ell_{\infty}(J)$ such that the set
$\big\{j\in J \,\big\vert\, x(j)\neq 0 \big\}$ is finite. 
For every $x\in\ell_{p}(J)$ the sum in (\ref{sumdef}) is finite 
so by considering the set inclusion $\subseteq$ as a partial order on
the set of all finite subsets $F$ of $J$, 
the set $\big\{\sum_{j\in F}\abs{x(j)}^{p} \big\}_{F}$ is a net in $\R$
that converges to $\sum_{j\in J}\abs{x(j)}^{p}=\norm{x}_{p}^{p}$
for all $1\leq p<\infty$. Therefore, it follows that 
$c_{00}(J)$ is dense 
in $\ell_{p}(J)$ for all $1\leq p<\infty$.

\section{The metric compactification of $\ell_{1}$}\label{horocompl1}
Throughout this section for each $y\in\ell_{1}(J)$ we denote by 
$h_{y}$ the function defined on $\ell_{1}(J)$ by
\begin{equation*}
	h_{y}(x):=\norm{x-y}_{1}-\norm{y}_{1}, \text{ for all } x\in\ell_{1}(J). 
\end{equation*}

In order to find possible metric functionals
on $\ell_{1}(J)$ we will need the following argument,
which was also used in \cite[Lemma~3.1]{Gutierrez2017} 
to describe the metric compactification 
of the finite-dimensional space $\ell_{1}(\{1,...,N\})$. 
\begin{prop}\label{1-dimhoro}
Let $\{a_{\beta}\}_{\beta}$ be a net of real numbers. Then for every
$r\in\R$ 
\begin{align*}
	\abs{r-a_{\beta}}-\abs{a_{\beta}}\underset{\beta}\longrightarrow
	\begin{cases}
		-r, &\text{ if } a_{\beta}\underset{\beta}\longrightarrow +\infty,\\
		r,  &\text{ if } a_{\beta}\underset{\beta}\longrightarrow -\infty.
	\end{cases}
\end{align*}
Moreover, since $\{\abs{\cdot-a_\beta}-\abs{a_{\beta}}\}_{\beta}$
is a family of $1$-Lipschitz functions on $\R$, the convergence above 
is uniform on compact subsets of $\R$. 
\end{prop}
\begin{remark}\label{sumC00}
For every $x\in c_{00}(J)$ there exists a finite
subset $F$ of $J$ such that $x(j)\neq 0$ for all
$j\in F$, and $x(j)=0$ for all $j\in J\setminus F$.
Therefore,
\begin{align*}
	h_{y}(x) &=\sum_{j\in J}\abs{x(j)-y(j)}-\sum_{j\in J}\abs{y(j)}\\
			&=\sum_{j\in F}\big( \abs{x(j)-y(j)}-\abs{y(j)} \big).
\end{align*}
\end{remark}
%*********************************************************************************
\begin{lem}\label{lemmaAl1}
Let $\{y_{\alpha}\}_{\alpha}$ be any net in 
$\ell_{1}(J)$. Then there exists a subnet $\{y_{\beta}\}_{\beta}$, 
a subset $I$ of $J$, a vector of signs $\varepsilon\in\{-1,+1\}^{I}$, 
and a vector $z\in\R^{J\setminus I}$ 
such that the net $\{h_{y_{\beta}}\}_{\beta}$ converges 
pointwise on $\ell_{1}(J)$ to the function
\begin{align*}
	x\mapsto h^{\{I,\varepsilon,z\}}(x):=\sum_{j\in I}\varepsilon(j)x(j) 
						+ \sum_{j\in J\setminus I}\left( \abs{x(j)-z(j)}-\abs{z(j)} \right).
\end{align*}
\end{lem}
\begin{proof}
Since $\ell_{1}(J)\subset \R^{J}\subset [-\infty,+\infty]^{J}$,
we may think of $\{y_{\alpha}\}_{\alpha}$ as
a net in the compact topological space 
$[-\infty,+\infty]^{J}$ with respect to the 
product topology. 
Therefore, there exists a subnet $\{y_{\beta}\}_{\beta}$
and a vector $\tilde{z}\in[-\infty,+\infty]^{J}$ such that
\begin{equation*}
	y_{\beta}(j)\underset{\beta}\longrightarrow \tilde{z}(j), \text{ for all } j\in J. 
\end{equation*}
Hence, there exists a subset $I$ 
(possibly empty) of $J$ such that $\tilde{z}(j)\in\{-\infty,+\infty\}$ 
for all $j\in I$, and $\tilde{z}(j)\in\R$ 
for all $j\in J\setminus I$. 
Put $z(j)=\tilde{z}(j)$ for all $j\in J\setminus I$, and 
for every $j\in I$
\begin{align*}
	\varepsilon(j)=\begin{cases}
						-1, &\text{ if } \tilde{z}(j)=+\infty,\\
						+1, &\text{ if } \tilde{z}(j)=-\infty.						
					\end{cases}
\end{align*}
Now, let $x$ be any element of $c_{00}(J)$. 
It follows from \Cref{sumC00} that there exists a finite
subset $F$ of $J$ such that $x(j)\neq 0$ for all
$j\in F$, and $x(j)=0$ for all $j\in J\setminus F$ for which
\begin{align*}
	h_{y_{\beta}}(x)  &= \sum_{j\in F\cap I}\left( \abs{x(j)-y_{\beta}(j)}-\abs{y_{\beta}(j)} \right) \\
					  & \qquad\qquad + \sum_{j\in F\cap (J\setminus I)}\left( \abs{x(j)-y_{\beta}(j)}-\abs{y_{\beta}(j)} \right).	
\end{align*}
By \Cref{1-dimhoro} we obtain
\begin{align*}
	\lim_{\beta}h_{y_{\beta}}(x) &= \sum_{j\in F\cap I}\varepsilon(j)x(j)
					 			 + \sum_{j\in F\cap (J\setminus I)}\left( \abs{x(j)-z(j)}-\abs{z(j)} \right)\\
								&= \sum_{j\in I}\varepsilon(j)x(j)
					  				+ \sum_{j\in J\setminus I}\left( \abs{x(j)-z(j)}-\abs{z(j)} \right).
\end{align*} 					  
Therefore, the claim of the Lemma follows readily from
the fact that $c_{00}(J)$ is dense in $\ell_{1}(J)$ and
$\{h_{y_{\beta}}\}_{\beta}$ is a family of $1$-Lipschitz
functions on $\ell_{1}(J)$. 
\end{proof}
\begin{example}
For simplicity let us assume that $J=\N$. Consider the 
following sequences in $\ell_{1}(\N)$:
\begin{align*}
	y_{n} &=(1,0,...,0,\underset{n^{\text{th}}}{n},0,0,...),\\
	\tilde{y}_{n} &=(n,1,1,...,\underset{n^{\text{th}}}{1},0,0,...).
\end{align*}
Then for every $x\in\ell_{1}(\N)$ we have
\begin{align}
	\lim_{n\to\infty}h_{y_{n}}(x) 
	&= \big( \abs{x(1)-1}-1\big) +\sum_{j=2}^{\infty}\abs{x(j)}=h_{z}(x),\label{ex1}\\
	\lim_{n\to\infty}h_{\tilde{y}_{n}}(x) 
	&= -x(1)+\sum_{j=2}^{\infty}\left( \abs{x(j)-1}-1 \right)=h^{\{I,\varepsilon,z\}}(x),\label{ex2}
\end{align}
where $z=(1,0,0,...)\in\ell_{1}(\N)$ in (\ref{ex1}) 
whereas $I=\{1\}\subset\N$, $\varepsilon(1)=-1$, $z(j)=1$ 
for all $j\geq 2$ in (\ref{ex2}).
Notice that in both cases we have
$\norm{y_{n}}_{1}\to\infty$ and $\norm{\tilde{y}_{n}}_{1}\to\infty$, 
as $n\to\infty$. However, $h^{\{I,\varepsilon,z\}}$ 
in (\ref{ex2}) is a metric functional at infinity
while $h_{z}$ in (\ref{ex1}) is an internal metric functional.
\end{example}
The following result shows that bounded
nets in $\ell_{1}(J)$ produce only internal
metric functionals. 
%********************************************************************************
\begin{lem}\label{lemmaBl1}
Let $\{y_{\alpha}\}_{\alpha}$ be a bounded net in 
$\ell_{1}(J)$. Then there exists a subnet 
$\{y_{\beta}\}_{\beta}$ and a vector $z\in\ell_{1}(J)$ 
such that the net $\{h_{y_{\beta}}\}_{\beta}$ 
converges pointwise to the function $h_{z}$.
\end{lem}
\begin{proof}
Recall that $\ell_{1}(J)$ and $c_{0}(J)^{*}$ 
are isometrically isomorphic under the surjective isometry 
$L: \ell_{1}(J) \to c_{0}(J)^{*}$, $y\mapsto L_{y}$ defined
by $L_{y}(x)=\sum_{j\in J} x(j)y(j)$, for all $x\in c_{0}(J)$. 
Therefore, we may consider 
the bounded net $\{y_{\alpha}\}_{\alpha}$ in $\ell_{1}(J)$ 
as a bounded net of continuous linear functionals on $c_{0}(J)$.
Hence, by Alaoglu's theorem there exists a subnet $\{y_{\beta}\}_{\beta}$ 
and a vector $z\in\ell_{1}(J)$ such that 
$y_{\beta}\overset{w^{*}}{\underset{\beta}\longrightarrow} z$.
In particular, we have $y_{\beta}(j)\underset{\beta}\longrightarrow z(j)$ 
for all $j\in J$. Now, let $x$ be any vector in $c_{00}(J)$. By
\Cref{sumC00}, there exists a finite
subset $F$ of $J$ such that $x(j)\neq 0$ for all
$j\in F$, and $x(j)=0$ for all $j\in J\setminus F$. Hence, 
\begin{align*}
	h_{y_{\beta}}(x) &= \sum_{j\in F}\abs{x(j)-y_{\beta}(j)}-\sum_{j\in F}\abs{y_{\beta}(j)}\\
					 & \underset{\beta}\longrightarrow \sum_{j\in F} \abs{x(j)-z(j)} - \sum_{j\in F}\abs{z(j)}.
\end{align*}
Therefore, since $z\in \ell_{1}(J)$ and $x(j)=0$ for all $j\in J\setminus F$, 
we obtain
\begin{equation*}
	\lim_{\beta}h_{y_{\beta}}(x) = \sum_{j\in J} \abs{x(j)-z(j)} - \sum_{j\in J}\abs{z(j)} = h_{z}(x).
\end{equation*}
Finally, since $c_{00}(J)$ is dense in $\ell_{1}(J)$
and $\{h_{y_{\beta}}\}_{\beta}$ is a family of $1$-Lipschitz functions 
on $\ell_{1}(J)$, the claim follows.
\end{proof}
%*********************************************************************************
\begin{thrm}\label{Thml1}
The metric compactification of $\ell_{1}(J)$ is given by
\begin{equation}\label{l1horoset}
		\overline{\ell_{1}(J)}^{h}=\left\{h^{\{I,\varepsilon,z\}}\in\R^{\ell_{1}(J)} \,\bigg\vert\,
						\begin{aligned}
						   \emptyset\subseteq  I \subseteq  & J,\\
						   \varepsilon\in\{-1,+1\}^{I},\;
						    & z\in\R^{J\setminus I} 
						\end{aligned}						    
						    \right\},
\end{equation}
where for every $x\in\ell_{1}(J)$
\begin{align*}
		h^{\{I,\varepsilon,z\}}(x) =\sum_{j\in I}\varepsilon(j)x(j) 
						+ \sum_{j\in J\setminus I}\left( \abs{x(j)-z(j)}-\abs{z(j)} \right).
\end{align*}
\end{thrm}
\begin{proof}
%[Proof of \Cref{Thml1}]
Let $h$ be any element of $\overline{\ell_{1}(J)}^{h}$. 
Then there exists a net $\{y_{\alpha}\}_{\alpha}$ in
$\ell_{1}(J)$ such that $h_{y_{\alpha}}\underset{\alpha}\longrightarrow h$
pointwise on $\ell_{1}(J)$.
By \Cref{lemmaAl1}, there exists a 
subnet $\{h_{y_{\beta}}\}_{\beta}$ of $\{h_{y_{\alpha}}\}_{\alpha}$, 
a subset $I$ of $J$, a vector of signs $\varepsilon\in \{-1,+1\}^{I}$,
and a vector $z\in\R^{J\setminus I}$ such that for every 
$x\in\ell_{1}(J)$
\begin{align*}
	\lim_{\beta}h_{y_{\beta}}(x)=
		h^{\{I,\varepsilon,z\}}(x)=\sum_{j\in I}\varepsilon(j)x(j) 
						+ \sum_{j\in J\setminus I}\big( \abs{x(j)-z(j)}-\abs{z(j)} \big).
\end{align*} 
Then $h=h^{\{I,\varepsilon,z\}}$ and hence $h$ belongs to the set 
on the right-hand side of (\ref{l1horoset}).

Suppose now that $h\in\R^{\ell_{1}(J)}$ is defined 
for every $x\in\ell_{1}(J)$ by 
\begin{align}\label{liml1}
	h(x)=\sum_{j\in I}\varepsilon(j)x(j) 
						+ \sum_{j\in J\setminus I}\big( \abs{x(j)-z(j)}-\abs{z(j)} \big),
\end{align}
for some subset $I$ of $J$, some $\varepsilon\in\{-1,+1\}^{I}$, and 
some $z\in\R^{J\setminus I}$.
Then for each finite subset $F$ (with cardinality $\#F$) of $J$ define 
\begin{align*}
		y_{F}(j):=\begin{cases}
						-\varepsilon(j)\#F, & \text{ if } j\in F\cap I,\\
						z(j), &	\text{ if } j\in F\cap(J\setminus I),\\
						0, &	\text{ if } j\in J\setminus F.
					\end{cases}	
\end{align*} 
With set inclusion $\subseteq$ as a partial order on the set of all finite subsets $F$
of $J$, the set $\{y_{F}\}_{F}$ defines a net in $\ell_{1}(J)$ with norm
\begin{equation*}
	\norm{y_{F}}_{1}=\sum_{j\in J}\abs{y_{F}(j)}
					=\sum_{j\in F\cap I} \#F
				 	+ \sum_{j\in F\cap(J\setminus I)}\abs{z(j)}.
\end{equation*} 
We claim that for every $x\in \ell_{1}(J)$ 
we have $h_{y_{F}}(x)\underset{F}\longrightarrow h(x)$, 
where $h$ is given in (\ref{liml1}). If we prove this
for all $x\in c_{00}(J)$, the claim follows readily from the 
fact that $c_{00}(J)$ is dense in $\ell_{1}(J)$ and
the net $\{h_{y_{F}}\}_{F}$ is a family of
$1$-Lipschitz functions on $\ell_{1}(J)$. Let $x$ be 
any element in $c_{00}(J)$. By \Cref{sumC00}, there
exists a finite set $G$ of $J$ such that
\begin{align}\label{sum0}
	h_{y_{F}}(x)&=\sum_{j\in G}\big( \abs{x(j)-y_{F}(j)}-\abs{y_{F}(j)}\big)\nonumber\\
				&=\sum_{j\in G\cap F\cap I}\big( \big\vert -\varepsilon(j)x(j)-\#F \big\vert - \#F \big)\\
				 &\qquad +\sum_{j\in G\cap F\cap(J\setminus I)} \big( \abs{x(j)-z(j)}-\abs{z(j)} \big)\nonumber
				  +\sum_{j\in G\cap (J\setminus F)}\abs{x(j)}.	
\end{align}
As $F$ gets larger, the third sum in (\ref{sum0}) converges
to $0$, while the second sum converges to 
$\sum_{j\in G\cap(J\setminus I)} \big( \abs{x(j)-z(j)}-\abs{z(j)} \big)$. 
By \Cref{1-dimhoro}, 
each term of the first sum in (\ref{sum0}) converges 
to $\varepsilon(j)x(j)$, as $F$ gets larger. Moreover, 
larger subsets $F$ of $J$ will eventually contain the finite set $G$. 
Hence, the first sum of (\ref{sum0}) converges to 
$\sum_{j\in G\cap I} \varepsilon(j)x(j)$, as $F$ gets larger.
We have therefore shown that for every $x\in c_{00}(J)$,
\begin{equation*}
	h_{y_{F}}(x)\underset{F}\longrightarrow 
		\sum_{j\in G\cap I} \varepsilon(j)x(j) + \sum_{j\in G\cap(J\setminus I)} \big( \abs{x(j)-z(j)}-\abs{z(j)} \big)
		= h(x),
\end{equation*}
where $G$ is the finite set $\{j\in J \mid x(j)\neq 0\}$. 
This concludes the proof of the theorem.
\end{proof}

By a simple inspection of all the elements of
the set (\ref{l1horoset}) we obtain the following.
\begin{cor}\label{corl1}
The set of finite metric functionals on $\ell_{1}(J)$ consists only of
internal ones. That is, $\,\overline{\ell_{1}(J)}^{h,F}=\{h_{y} \,\big\vert\, y\in \ell_{1}(J) \}$.
\end{cor}

We will see in the next sections that \Cref{corl1} does not hold 
for infinite-dimensional Hilbert spaces and $\ell_{p}(J)$ with $1<p<\infty$.
%********************************************************************************************************************
\section{The metric compactification of infinite-dimensional Hilbert spaces}\label{horoHilbert}
Throughout this section we will assume that
$\mathcal{H}$ is an infinite-dimensional real Hilbert space with inner product $\iprod{\cdot,\cdot}$
and norm $\norm{x}=\iprod{x,x}^{1/2}$ for all $x\in\mathcal{H}$. 
For each $y\in\mathcal{H}$ we denote by $h_{y}$ the function on 
$\mathcal{H}$ defined by
\[
	x\mapsto h_{y}(x):=\norm{x-y}-\norm{y}.
\]
%************************************************************************************
\begin{lem}\label{horoHilbA}
Let $\{y_{\alpha}\}_{\alpha}$ be a bounded net in 
$\mathcal{H}$. Then there exists a subnet 
$\{y_{\beta}\}_{\beta}$, a vector $z\in\mathcal{H}$
and a real number $c\geq\norm{z}$
such that the net $\{h_{y_{\beta}}\}_{\beta}$ 
converges pointwise on $\mathcal{H}$ to the function 
\begin{align}
	x\mapsto h^{\{z,c\}}(x):=\big( \norm{x}^{2}-2\iprod{x,z}+c^{2} \big)^{1/2} - c.
\end{align}
\end{lem}
\begin{proof}
Let $\{y_{\alpha}\}_{\alpha}$ be a bounded net in 
$\mathcal{H}$. Since $\mathcal{H}$ is reflexive, 
by Alaoglu's theorem there exists a subnet 
$\{y_{\beta}\}_{\beta}$ and a vector $z\in\mathcal{H}$
such that
\begin{equation}\label{weaklimH}
	y_{\beta}\overset{w}{\underset{\beta}\longrightarrow} z.
\end{equation}
By letting $c:=\liminf_{\beta}\norm{y_{\beta}}$
we have $\norm{z}\leq c$. Now, 
since $\{\norm{y_{\beta}}\}_{\beta}$ is 
a bounded net in $\R$, by passing to a subnet
we may assume that
\begin{equation}\label{limnorm}
	\lim_{\beta}\norm{y_{\beta}} = c. 
\end{equation}
Let $x$ be any element in $\mathcal{H}$. Then, by (\ref{weaklimH}) and 
(\ref{limnorm}), we obtain
\begin{align*}
	h_{y_{\beta}}(x) &= \norm{x-y_{\beta}}-\norm{y_{\beta}} \\
		 			&= \big( \norm{x}^{2}-2\iprod{x,y_{\beta}}
		 						+\norm{y_{\beta}}^{2} \big)^{1/2} - \norm{y_{\beta}} \\
		 			&\underset{\beta}\longrightarrow 
		 				\left( \norm{x}^{2}-2\iprod{x,z}+c^{2} \right)^{1/2} - c
		 			\;= h^{\{z,c\}}(x).
\end{align*}
\end{proof}
\begin{remark}[Radon-Riesz property]\label{RRprop}
The weak limit (\ref{weaklimH}) and the limit (\ref{limnorm}) 
of the norms $\norm{y_{\beta}}$ in \Cref{horoHilbA} imply
\begin{align*}
	\norm{y_{\beta}-z}^{2} &= \norm{y_{\beta}}^{2}-2\iprod{y_{\beta},z}+\norm{z}^{2}\\
							&\underset{\beta}\longrightarrow c^{2}-2\iprod{z,z}+\norm{z}^{2}\\
							&= c^{2}-\norm{z}^{2}.
\end{align*}
Hence $y_{\beta}$ converges
strongly to $z$ if and only if 
$c=\norm{z}$. It is clear that $h^{\{z,\norm{z}\}}=h_{z}$.
\end{remark}
%*******************************************************************************************
\begin{lem}\label{horoHilbB}
Let $\{y_{\alpha}\}_{\alpha}$ be a net in $\mathcal{H}$ such that
$\norm{y_{\alpha}}\underset{\alpha}\longrightarrow +\infty$.
Then there exists a subnet $\{y_{\beta}\}_{\beta}$ and
a vector $z\in\mathcal{H}$ with $\norm{z}\leq 1$
such that the net $\{h_{y_{\beta}}\}_{\beta}$ converges
pointwise on $\mathcal{H}$ to the function
\begin{align}
	x\mapsto h^{\{z\}}(x):=-\iprod{x,z}.
\end{align}
\end{lem}
\begin{proof}
Since $\norm{y_{\alpha}}\underset{\alpha}\longrightarrow +\infty$,
by passing to a subnet we may assume that 
$\norm{y_{\alpha}} > 0$ for all $\alpha$.
Define $z_{\alpha}:= y_{\alpha}/\norm{y_{\alpha}}$ 
for all $\alpha$. Since  $\mathcal{H}$ is reflexive,
it follows from Alaoglu's theorem that 
there exists a subnet $\{z_{\beta}\}_{\beta}$ 
and a vector $z\in\mathcal{H}$ with $\norm{z}\leq 1$ such that
$z_{\beta} \overset{w}{\underset{\beta}\longrightarrow} z$. 
Let $x$ be any vector in $\mathcal{H}$. Then 
\begin{equation}\label{normxy}
\begin{aligned}
	\norm{x-y_{\beta}} &= \big( \norm{x}^{2}-2\iprod{x,y_{\beta}}+\norm{y_{\beta}}^{2} \big)^{1/2}\\
					   &= \norm{y_{\beta}} \bigg( \frac{\norm{x}^{2}}{\norm{y_{\beta}}^{2}}
						-\frac{2\iprod{x,z_{\beta}}}{\norm{y_{\beta}}}+1 \bigg)^{1/2}.		
\end{aligned}							   
\end{equation} 
It is clear that $\norm{y_{\beta}}\underset{\beta}\longrightarrow +\infty$
implies
\begin{equation}\label{smalla}
	\frac{\norm{x}^{2}}{\norm{y_{\beta}}^{2}} - \frac{2\iprod{x,z_{\beta}}}{\norm{y_{\beta}}} \underset{\beta}\longrightarrow 0.
\end{equation}
Recall the Taylor expansion 
$\sqrt{a + 1}=1+a/2+ O(a^{2})$, when $a$ is small 
enough. Hence by (\ref{smalla}) we can write 
(\ref{normxy}) as follows
\begin{align*}
	\norm{x-y_{\beta}} &= \norm{y_{\beta}}\bigg( 1 + \frac{1}{2}
							\bigg(\frac{\norm{x}^{2}}{\norm{y_{\beta}}^{2}}
							-\frac{2\iprod{x,z_{\beta}}}{\norm{y_{\beta}}}\bigg)
					    	+ O\bigg(\frac{1}{\norm{y_{\beta}}^{2}}\bigg)\bigg)\\
				   &= \norm{y_{\beta}} + \frac{1}{2}\bigg(\frac{\norm{x}^{2}}
				   		{\norm{y_{\beta}}}-2\iprod{x,z_{\beta}}\bigg) 
				   		+ O\bigg(\frac{1}{\norm{y_{\beta}}}\bigg),	    	
\end{align*}
as $\norm{y_{\beta}}\underset{\beta}\longrightarrow +\infty$.
Therefore, we obtain
\begin{align*}
	 \lim_{\beta} h_{y_{\beta}}(x) &= \lim_{\beta} \big[ \norm{x-y_{\beta}}- \norm{y_{\beta}}\big]\\
				&= \lim_{\beta} \bigg[ \frac{1}{2}\bigg(\frac{\norm{x}^{2}}
				   		{\norm{y_{\beta}}}-2\iprod{x,z_{\beta}}\bigg) 
				   		+ O\bigg(\frac{1}{\norm{y_{\beta}}}\bigg) \bigg]\\ 
				& = - \iprod{x,z} \;= h^{\{z\}}(x).
\end{align*}
\end{proof}
%*******************************************************************************************
\begin{thrm}\label{Hilberthoro}
The metric compactification 
$\overline{\mathcal{H}}^{h}=\overline{\mathcal{H}}^{h,F}\sqcup\overline{\mathcal{H}}^{h,\infty}$
of an infinite-dimensional real Hilbert space 
$\mathcal{H}$ is given by 
\begin{align*}
	\overline{\mathcal{H}}^{h,F} &= \left\{
						 h^{\{z,c\}}\in\R^{\mathcal{H}}
							\,\mathrel{\big\vert}\,	
								z\in \mathcal{H},\;c\geq\norm{z}	
						\right\}\cup\{0\}, \\
	\overline{\mathcal{H}}^{h,\infty} &= \left\{ 
						h^{\{z\}}\in\R^{\mathcal{H}}
							\,\mathrel{\big\vert}\,						
								z\in \mathcal{H},\; 0<\norm{z}\leq 1
						\right\},
\end{align*}
where for every $x\in\mathcal{H}$
\begin{align*}
		h^{\{z,c\}}(x) &= \big( \norm{x}^{2}-2\iprod{x,z}+c^{2} \big)^{1/2} - c, \\
		h^{\{z\}}(x) &= -\iprod{x,z}.
\end{align*}
\end{thrm}
%*********************************************************************************************
\begin{proof}
%[Proof of Theorem \ref{Hilberthoro}]
Let $h$ be any element of $\overline{\mathcal{H}}^{h}$. 
Then there exists a net $\{y_{\alpha}\}_{\alpha}$ 
in $\mathcal{H}$ such that 
$h_{y_{\alpha}}\underset{\alpha}\longrightarrow h$
pointwise on $\mathcal{H}$. Suppose that the net $\{y_{\alpha}\}_{\alpha}$
is bounded in $\mathcal{H}$. Then it follows from
\Cref{horoHilbA} that there exists a subnet 
$\{y_{\beta}\}_{\beta}$, a vector $z\in\mathcal{H}$,
and a real number $c\geq\norm{z}$
such that 
\begin{align*}
	\lim_{\beta} h_{y_{\beta}}(x) = \big( \norm{x}^{2}-2\iprod{x,z}+c^{2} \big)^{1/2} - c = h^{\{z,c\}}(x),
\end{align*}
for all $x\in\mathcal{H}$. Therefore $h=h^{\{z,c\}}$ 
and hence $h\in\overline{\mathcal{H}}^{h,F}\setminus\{0\}$. Suppose now that the
net $\{y_{\alpha}\}_{\alpha}$ is unbounded in $\mathcal{H}$. Then,
by passing to a subnet we may assume that 
$\norm{y_{\alpha}}\underset{\alpha}\longrightarrow +\infty$.
It follows from \Cref{horoHilbB} that there exists a subnet 
$\{y_{\beta}\}_{\beta}$ and a vector $z\in\mathcal{H}$ 
with $\norm{z}\leq 1$ such that 
\begin{equation*}
	\lim_{\beta} h_{y_{\beta}}(x) = -\iprod{x,z} = h^{\{z\}}(x), 
\end{equation*}
for all $x\in\mathcal{H}$. Hence $h=h^{\{z\}}$ and so 
$h\in\overline{\mathcal{H}}^{h,\infty}\cup\{0\}$. Consequently, 
we have proved that the 
inclusion
$\overline{\mathcal{H}}^{h}\subseteq\overline{\mathcal{H}}^{h,F}\sqcup\overline{\mathcal{H}}^{h,\infty}$
holds.

On the other hand, since $\mathcal{H}$ is infinite-dimensional
and reflexive, there exists a sequence $\{u_{n}\}_{n\in\N}$ 
in $\mathcal{H}$ with $\norm{u_{n}}=1$ for all $n\in\N$ such that 
$u_{n} \overset{w}{\longrightarrow} 0$, as $n\to\infty$. 
Suppose that $h^{\{z,c\}}\in\overline{\mathcal{H}}^{h,F}\setminus\{0\}$ 
for some $z\in\mathcal{H}$ and some $c\geq\norm{z}$. Then 
for each $n\in\N$ define $y_{n}\in\mathcal{H}$ by
\[
	y_{n}:=\big(c^{2}-\norm{z}^{2}\big)^{1/2} u_{n} + z.
\]
It follows that $y_{n} \overset{w}{\longrightarrow} z$ and also
\begin{equation*}
	\norm{y_{n}}^{2}=c^{2}+2\big(c^{2}-\norm{z}^{2}\big)^{1/2}\iprod{z,u_{n}} 
	\longrightarrow c^{2},
\end{equation*} 
as $n\to\infty$. Therefore, for every $x\in\mathcal{H}$
\begin{align*}
	h_{y_{n}}(x)  &= \norm{x-y_{n}} - \norm{y_{n}}\\
					&= \big( \norm{x}^{2}-2\iprod{x,y_{n}}+\norm{y_{n}}^{2} \big)^{1/2} - \norm{y_{n}}\\
					&\underset{n\to\infty}\longrightarrow 
					\big( \norm{x}^{2}-2\iprod{x,z}+c^{2} \big)^{1/2} - c\\
					&= h^{\{z,c\}}(x).
\end{align*}	
Hence $h^{\{z,c\}}$ is an element of $\overline{\mathcal{H}}^{h}$. 
Now, suppose that $h\in\overline{\mathcal{H}}^{h,\infty}\cup\{0\}$. 
Then for every $x\in \mathcal{H}$ we have $h(x)=-\iprod{x,z}$
for some  $z\in\mathcal{H}$ with $\norm{z}\leq 1$. For each 
$n\in\N$ define $\tilde{y}_{n}\in\mathcal{H}$ by
\[
	\tilde{y}_{n}:=n\big(1-\norm{z}^{2}\big)^{1/2} u_{n} + nz. 
\]
Then $\norm{\tilde{y}_{n}}^{2}=n^{2}\big(1+2\big(1-\norm{z}^{2}\big)^{1/2}\iprod{z,u_{n}}\big) 
	\longrightarrow \infty$, as $n\to\infty$.
Furthermore, we have $\tilde{y}_{n}/\norm{\tilde{y}_{n}} \overset{w}\longrightarrow z$, 
as $n\to\infty$. Then by proceeding as in (\ref{normxy}) and (\ref{smalla})
we obtain
\begin{equation*}
	h_{\tilde{y}_{n}}(x) \underset{n\to\infty}\longrightarrow -\iprod{x,z} = h(x),
\end{equation*}	
for all $x\in\mathcal{H}$ so $h$ 
is an element of $\overline{\mathcal{H}}^{h}$. We therefore have proved 
that the inclusion 
$\overline{\mathcal{H}}^{h,F}\sqcup\overline{\mathcal{H}}^{h,\infty}\subseteq\overline{\mathcal{H}}^{h}$
also holds.
\end{proof}
%*******************************************************************************************
\begin{remark}
It readily follows from \Cref{RRprop} that 
$h^{\{z,c\}}$ is an internal metric functional
if and only if $c=\norm{z}$. \Cref{Hilberthoro} 
states that metric functionals on an infinite-dimensional 
real Hilbert space $\mathcal{H}$ are of three types: 
\begin{enumerate}
\item The set $\{h_{y} \,\vert\, y\in\mathcal{H}\}$ 
		contains metric functionals which correspond 
		to strongly convergent nets in $\mathcal{H}$.
\item The set $\left\{ h^{\{z,c\}}\in\R^{\mathcal{H}}
							\,\mathrel{\big\vert}\,	
								z\in \mathcal{H},\;c>\norm{z}	
						\right\}$
		contains \emph{exotic} metric functionals which correspond
		to bounded nets converging weakly but not strongly in $\mathcal{H}$.
\item The set  $\overline{\mathcal{H}}^{h,\infty}\cup\{0\}$
		contains metric functionals which correspond to
		nets in $\mathcal{H}$ with norm tending to infinity.
\end{enumerate}
\end{remark}
%********************************************************************************************
\begin{remark}
For the infinite-dimensional 
real Hilbert space $\mathcal{H}=\ell_{2}(J)$ 
with norm $\norm{\cdot}_{2}$, it follows readily from \Cref{Hilberthoro} 
that the metric compactification 
$\overline{\ell_{2}(J)}^{h} = \overline{\ell_{2}}^{h,F} \sqcup \overline{\ell_{2}}^{h,\infty}$
is given by
\begin{align*}
	\overline{\ell_{2}}^{h,F} &= \left\{ 
						h^{\{z,c\}}\in\R^{\ell_{2}(J)}
							\;\big\vert\; z\in \ell_{2}(J),\; c\geq \norm{z}_{2}
						\right\}\cup\{0\}, \\
	\overline{\ell_{2}}^{h,\infty}  &= \left\{
						h^{\{z\}}\in\R^{\ell_{2}(J)}
							\;\big\vert\; z\in\ell_{2}(J),\; 0<\norm{z}_{2}\leq 1
						\right\},
\end{align*}
where for every $x\in\ell_{2}(J)$
\begin{align*}
	h^{\{z,c\}}(x) &=\big( \norm{x-z}_2^{2}+c^{2}-\norm{z}_2^{2} \big)^{1/2} - c, \\
	    h^{\{z\}}(x) &=-\sum_{j\in J}x(j)z(j). 
\end{align*}
This observation will be useful in \Cref{horolp} 
to identify all the metric functionals
on the remaining $\ell_{p}$ spaces.
\end{remark}
%*********************************************************************************************
\section{The metric compactification of $\ell_{p}$, with $1<p<\infty$}\label{horolp}
Throughout this section we assume $1<p<\infty$. 
For each $y\in\ell_{p}(J)$ we denote by 
$h_{y}$ the function on $\ell_{p}(J)$ given by
\begin{align*}
		x\mapsto h_{y}(x):=\norm{x-y}_{p}-\norm{y}_{p}.
\end{align*} 
%***********************************************************************************************
\begin{lem}\label{lphoroA}
Let $\{y_{\alpha}\}_{\alpha}$ be a bounded net in $\ell_{p}(J)$.
Then there exists a subnet $\{y_{\beta}\}_{\beta}$, a vector
$z\in\ell_{p}(J)$, and a real number $c\geq\norm{z}_{p}$
such that the net $\{h_{y_{\beta}}\}_{\beta}$ converges
pointwise on $\ell_{p}(J)$ to the function
\begin{equation*}
	x\mapsto h^{\{z,c\}}(x):=\big( \norm{x-z}_p^{p}+c^{p}-\norm{z}_p^{p} \big)^{1/p} - c.
\end{equation*} 
\end{lem}
\begin{proof}
Since $\{y_{\alpha}\}_{\alpha}$ is a bounded net in $\ell_{p}(J)$
and $\ell_{p}(J)$ is reflexive, it follows from Alaoglu's
theorem that there exists a subnet $\{y_{\beta}\}_{\beta}$, and a
vector $z\in\ell_{p}(J)$ such that $y_{\beta} \overset{w}{\underset{\beta}\longrightarrow} z$. 
In particular, we have 
\begin{equation}\label{coordconvp}
	y_{\beta}(j)\underset{\beta}\longrightarrow z(j), \;\text{ for all } j\in J. 
\end{equation}
By letting $c:=\liminf_{\beta}\norm{y_{\beta}}_{p}$ 
we have $\norm{z}_{p}\leq c$. Moreover, since $\{\norm{y_{\beta}}_{p}\}_{\beta}$ 
is a bounded net in $\R$, by passing to a subnet we may assume that
\begin{equation}\label{normconvp}
	\lim_{\beta}\norm{y_{\beta}}_{p} = c. 
\end{equation}

Let $x$ be any vector in $c_{00}(J)$. Then there 
exists a finite subset $F$ of $J$ such that
$x(j)\neq 0$ for all $j\in F$, and $x(j)=0$ otherwise.
By applying (\ref{coordconvp}) and (\ref{normconvp}), 
we obtain
\begin{align*}
	\norm{x-y_{\beta}}_{p}^{p} &= \sum_{j\in J}\abs{x(j)-y_{\beta}(j)}^{p}\\
							&= \sum_{j\in F}\abs{x(j)-y_{\beta}(j)}^{p} 
							+\sum_{j\in J\setminus F}\abs{y_{\beta}(j)}^{p}\\
							&=\sum_{j\in F}\abs{x(j)-y_{\beta}(j)}^{p} 
							+ \sum_{j\in J}\abs{y_{\beta}(j)}^{p} - \sum_{j\in F}\abs{y_{\beta}(j)}^{p}\\
							&\underset{\beta}\longrightarrow 
							\sum_{j\in F}\abs{x(j)-z(j)}^{p} 
							+ c^{p} - \sum_{j\in F}\abs{z(j)}^{p} \\
							&=\sum_{j\in J}\abs{x(j)-z(j)}^{p} 
							+ c^{p} - \sum_{j\in J}\abs{z(j)}^{p}\\
							&=\norm{x-z}_{p}^{p}+c^{p}-\norm{z}_{p}^{p}.
\end{align*}
Therefore, for every $x\in c_{00}(J)$ we have
\begin{align*}
	h_{y_{\beta}}(x) &= \norm{x-y_{\beta}}_{p}-\norm{y_{\beta}}_{p}\\
					&= \big( \norm{x-y_{\beta}}_{p}^{p} \big)^{1/p}-\norm{y_{\beta}}_{p}\\
					&\underset{\beta}\longrightarrow \big( \norm{x-z}_{p}^{p}+c^{p}-\norm{z}_{p}^{p}\big)^{1/p} - c\\
					&= h^{\{z,c\}}(x). 
\end{align*}
Since $c_{00}(J)$ is dense in $\ell_{p}(J)$, and
the set $\{h_{y_{\beta}}\}_{\beta}$ is a family of
$1$-Lipschitz functions on $\ell_{p}(J)$, it follows 
that $h_{y_{\beta}}\underset{\beta}\longrightarrow h^{\{z,c\}}$ 
pointwise on $\ell_{p}(J)$.
\end{proof}
%***************************************************************************************************
\Cref{lphoroA} describes metric functionals on $\ell_{p}(J)$
which correspond to bounded nets in $\ell_{p}(J)$. 
We now characterize possible 
metric functionals on $\ell_{p}(J)$ which correspond to nets with $p$-norm 
tending to infinity. In fact, we go further and give a characterization
of metric functionals which are obtained from nets with norm 
tending to infinity in any Banach space with uniformly convex dual. 
Afterwards we turn to the specific case of $\ell_{p}(J)$,
where we give a full characterization of its metric
compactification.

A Banach space $(V,\norm{\cdot})$ 
is called \textit{uniformly convex} if for every 
$0<\epsilon\leq 2$ there exists $\delta(\epsilon)>0$ 
such that $\norm{u+v}\leq 2(1-\delta)$ whenever 
$u,v\in V$ with $\norm{u}=\norm{v}=1$
and $\norm{u-v}\geq\epsilon$. A well-known result 
due to Clarkson \cite{Clarkson1936} is that $L_{p}$ 
and $\ell_{p}$ spaces are uniformly convex for all $1<p<\infty$. 

\begin{remark}\label{uniconvex}
If $\{u_{\beta}\}_{\beta}$ and $\{v_{\beta}\}_{\beta}$ are nets 
in the unit sphere of a uniformly convex 
Banach space $(V,\norm{\cdot})$ such that
$\norm{u_{\beta}+v_{\beta}}\underset{\beta}\longrightarrow 2$,
then we have $\norm{u_{\beta}-v_{\beta}}\underset{\beta}\longrightarrow 0$.
Indeed, if we suppose that
$\norm{u_{\beta}-v_{\beta}}\underset{\beta}{\centernot\longrightarrow} 0$,
then by passing to a subnet we may assume that 
$\norm{u_{\beta}-v_{\beta}}\geq\epsilon$ for some $\epsilon>0$ 
and all $\beta$. By uniform convexity, there exists
$\delta(\epsilon)>0$ such that $\norm{u_{\beta}+v_{\beta}}\leq 2(1-\delta)$,
for all $\beta$. Therefore, 
\begin{align*}
	2=\lim_{\beta}\norm{u_{\beta}+v_{\beta}}\leq 2-2\delta,
\end{align*}
which is a contradiction.
\end{remark}
%**********************************************************************************************
\begin{lem}\label{unblp}
Let $(X,\norm{\cdot})$ be an infinite-dimensional 
Banach space with uniformly convex dual space $(X^{*},\norm{\cdot}_{*})$. 
Let $\{y_{\alpha}\}_{\alpha}$ be a net in $X$ 
with $\norm{y_{\alpha}}\underset{\alpha}\longrightarrow +\infty$. 
Then there exists a subnet $\{y_{\beta}\}_{\beta}$ and a continuous linear
functional $\mu\in X^{*}$ with $\norm{\mu}_*\leq 1$
such that $h_{y_{\beta}}\underset{\beta}\longrightarrow -\mu$ 
pointwise on $X$.
\end{lem}
\begin{proof}
Without loss of generality we may assume that $\norm{y_{\alpha}} \neq 0$, 
for all $\alpha$. Consider the net of unit vectors 
$\{z_{\alpha}\}_{\alpha}$ in $X$ given by 
$z_\alpha = y_{\alpha}/\norm{y_{\alpha}}$, for all $\alpha$. 
By the Hahn-Banach theorem, for each $\alpha$ there exists 
$\mu_{\alpha}\in X^{*}$ with $\norm{\mu_\alpha}_*=1$ 
such that $\mu_{\alpha}(z_{\alpha})=1$.
Hence, by Alaoglu's theorem there exists
a subnet $\{\mu_{\beta}\}_{\beta}$ and a continuous linear
functional $\mu\in X^{*}$ 
with $\norm{\mu}_{*}\leq 1$ such that 
$\mu_{\beta} \overset{w^{*}}{\underset{\beta} \longrightarrow} \mu$.
Now, let $x$ be any vector in $X$. By extracting a further subnet 
if necessary, we may assume that $x-y_{\beta}\neq 0$ for all $\beta$. 
Define the net of unit vectors $\{z_{\beta}^{x}\}_{\beta}$ in $X$ by  
\begin{equation}\label{unitzk}
	z_{\beta}^{x}:=\frac{y_{\beta} - x}{\norm{x-y_{\beta}}}.
\end{equation}
By the Hahn-Banach theorem, it follows that for each $\beta$
there exists $\mu_{\beta}^{x}\in X^{*}$ with 
$\norm{\mu_{\beta}^{x}}_*=1$ such that
\begin{equation}\label{funatx}
		1=\mu_{\beta}^{x}(z_{\beta}^{x})
			=\dfrac{\mu_{\beta}^{x}(y_{\beta})-\mu_{\beta}^{x}(x)}{\norm{x-y_{\beta}}}
			=\dfrac{\mu_{\beta}^{x}(z_{\beta})-\mu_{\beta}^{x}\bigg(\dfrac{x}{\norm{y_{\beta}}}\bigg)}
					{\norm{\dfrac{x}{\norm{y_{\beta}}}-z_\beta}}.				
\end{equation}
On the other hand, for every $\beta$ we have
\begin{equation}\label{sandw}
	-\frac{\norm{x}}{\norm{y_\beta}} +1 \leq \norm{\frac{x}{\norm{y_\beta}}-z_\beta}
								\leq \frac{\norm{x}}{\norm{y_\beta}} + 1,
\end{equation}
and also
\begin{equation}\label{zeroinq}
	\abs{\mu_{\beta}^{x}\bigg(\dfrac{x}{\norm{y_{\beta}}}\bigg)}
								\leq \dfrac{\norm{x}}{\norm{y_{\beta}}}.
\end{equation}
Hence, by applying the assumption $\norm{y_{\beta}} \underset{\beta}\longrightarrow +\infty$ 
in (\ref{sandw}) and (\ref{zeroinq}) we obtain 
\begin{align}
	\norm{\frac{x}{\norm{y_\beta}}-z_\beta} &\underset{\beta}\longrightarrow 1, \label{limitone} \\
	\mu_{\beta}^{x}\bigg(\dfrac{x}{\norm{y_{\beta}}}\bigg) &\underset{\beta}\longrightarrow 0. \label{zero}
\end{align} 
Therefore, by applying (\ref{limitone}) and (\ref{zero}) in 
(\ref{funatx}), it follows that
\begin{equation}\label{limitXone}
	\mu_{\beta}^{x}(z_{\beta})\underset{\beta}\longrightarrow 1, 
\end{equation}
and hence, 
\[
	2\geq\norm{\mu_{\beta}^{x}+\mu_{\beta}}_{*}
		\geq \abs{(\mu_{\beta}^{x}+\mu_{\beta})(z_{\beta})}
		=\abs{\mu_{\beta}^{x}(z_{\beta})+\mu_{\beta}(z_{\beta})}
	\underset{\beta}\longrightarrow 2.
\]
Since $X^{*}$ is uniformly convex, it follows from \Cref{uniconvex}
that 
\begin{equation*}
	\norm{\mu_{\beta}^{x}-\mu_{\beta}}_*\underset{\beta}\longrightarrow 0.
\end{equation*}
Then we obtain
\begin{equation}\label{functconv}
	\mu_{\beta}^{x}\overset{w^{*}}{\underset{\beta}\longrightarrow} \mu.
\end{equation}
Finally, it follows from (\ref{funatx}) that
\begin{equation*}
	\norm{x-y_{\beta}}=-\mu_{\beta}^{x}(x)+\mu_{\beta}^{x}(z_{\beta})\norm{y_{\beta}},
\end{equation*}
and therefore by applying (\ref{limitXone}) and (\ref{functconv}) we obtain
\begin{align*}
	h_{y_{\beta}}(x)&=\norm{x-y_{\beta}}-\norm{y_{\beta}}\\
						&=-\mu_{\beta}^{x}(x)+\mu_{\beta}^{x}(z_{\beta})\norm{y_{\beta}}-\norm{y_{\beta}}\\
						&\underset{\beta}\longrightarrow -\mu(x).
\end{align*}
\end{proof}

We are now ready to give a full characterization 
of the metric compactification of $\ell_{p}(J)$
for all $1<p<\infty$.
%***************************************************************************************
\begin{thrm}\label{Thmlp}
Let $1<p,q<\infty$ such that $p^{-1}+q^{-1}=1$. 
The metric compactification
$\overline{\ell_{p}(J)}^{h}=\overline{\ell_{p}(J)}^{h,F}\sqcup\overline{\ell_{p}(J)}^{h,\infty}$ 
of the infinite-dimensional $\ell_{p}(J)$ space is given by
\begin{align*}
	\overline{\ell_{p}(J)}^{h,F} &= \left\{
						h^{\{z,c\}}\in\R^{\ell_{p}(J)}			
								\;\big\vert\; z\in \ell_{p}(J),\; c\geq\norm{z}_{p}
						\right\}\cup\{0\}, \\
	\overline{\ell_{p}(J)}^{h,\infty} &= \left\{ 
						h^{\{\mu\}}\in\R^{\ell_{p}(J)} \;\big\vert\; 						
								\mu\in \ell_{q}(J),\; 0<\norm{\mu}_{q}\leq 1
						\right\},
\end{align*}
where for every $x\in\ell_{p}(J)$
\begin{align*}
	h^{\{z,c\}}(x) &=\big( \norm{x-z}_{p}^{p}+c^{p}-\norm{z}_{p}^{p} \big)^{1/p} - c,\\
	h^{\{\mu\}}(x) &=-\sum_{j\in J}\mu(j)x(j).
\end{align*}
\end{thrm}
%*******************************************************************************************
\begin{proof}
%[Proof of Theorem \ref{Thmlp}]
Let $h\in\overline{\ell_{p}(J)}^{h}$. 
Then there exists a net $\{y_{\alpha}\}_{\alpha}$ 
in $\ell_{p}(J)$ such that 
$h_{y_{\alpha}}\underset{\alpha}\longrightarrow h$ pointwise on $\ell_{p}(J)$. 
Suppose that the net $\{y_{\alpha}\}_{\alpha}$ is bounded in $\ell_{p}(J)$. 
Then it follows from \Cref{lphoroA} that there exists a 
subnet $\{y_{\beta}\}_{\beta}$, a vector
$z\in\ell_{p}(J)$, and a real number $c\geq\norm{z}_{p}$
such that 
\begin{equation*}
	\lim_{\beta} h_{y_{\beta}}(x)= \big( \norm{x-z}_p^{p}+c^{p}-\norm{z}_p^{p} \big)^{1/p} - c
						=h^{\{z,c\}}(x),\; \text{ for all } x\in\ell_{p}(J). 
\end{equation*} 
Hence $h=h^{\{z,c\}}$ 
and so $h\in \overline{\ell_{p}(J)}^{h,F}\setminus\{0\}$.
Now, suppose that the net $\{y_{\alpha}\}_{\alpha}$ is unbounded in 
$\ell_{p}(J)$. Then, by passing to a subnet we may assume that
$\norm{y_{\alpha}}_{p}\underset{\alpha}\longrightarrow +\infty$. 
Since the dual space of $\ell_{p}(J)$ is the uniformly convex space $\ell_{q}(J)$,
it follows from \Cref{unblp} 
and $\ell_{p}/\ell_{q}$ duality that there exists 
$\mu\in \ell_{q}(J)$ with $\norm{\mu}_{q}\leq 1$ such that
\begin{equation*}
	\lim_{\beta} h_{y_{\beta}}(x)= -\sum_{j\in J}\mu(j)x(j), 
\end{equation*}
for all $x\in\ell_{p}(J)$. Therefore $h$ belongs to 
$\overline{\ell_{p}(J)}^{h,\infty}\cup\{0\}$. We have proved that
the inclusion
$\overline{\ell_{p}(J)}^{h} \subseteq \overline{\ell_{p}(J)}^{h,F}\sqcup\overline{\ell_{p}(J)}^{h,\infty}$
holds.

On the other hand, suppose that $h^{\{z,c\}}\in\overline{\ell_{p}(J)}^{h,F}\setminus\{0\}$, 
for some $z\in\ell_{p}(J)$ and some $c\geq\norm{z}_p$.
Put $a_{z,c}=\left( c^{p}-\norm{z}_{p}^{p} \right)^{1/p}$.
Pick any countably infinite subset $K=\{j_{m}\}_{m=1}^{\infty}$ 
of $J$, and for each $m\in\N$ define
\begin{equation}\label{bdseq}
	y_{m}(j):=\begin{cases}
         a_{z,c} + z(j), & \text{if } j=j_{m},\\
         z(j), & \text{if } j\neq j_{m}.
        \end{cases}
\end{equation}
It is clear that $y_{m}\in\ell_{p}(J)$ for all $m\in\N$. 
Moreover, we have $y_{m}\overset{w}\longrightarrow z$ and 
$\norm{y_{m}}_{p}^{p}\longrightarrow c^{p}$, as $m\to\infty$.
Then, for every $x\in\ell_{p}(J)$ we have
\begin{align*}
	\norm{x-y_{m}}_{p}^{p} &= \sum_{j\neq j_{m}}\abs{x(j)-z(j)}^{p} 
							+ \abs{x(j_{m})- a_{z,c} - z(j_{m})}^{p}\\
							&= \norm{x-z}_{p}^{p}-\abs{x(j_{m})-z(j_{m})}^{p}
							+ \abs{x(j_{m})- a_{z,c} - z(j_{m})}^{p}\\
							&\underset{m\to\infty}\longrightarrow 
							\norm{x-z}_{p}^{p}+c^{p}-\norm{z}_{p}^{p}.
\end{align*}
Therefore, (\ref{bdseq}) defines a bounded sequence in $\ell_{p}(J)$ such that
\begin{align*}
	\lim_{m\to\infty}h_{y_{m}}(x) &=\lim_{m\to\infty} \big[ \norm{x-y_{m}}_{p} - \norm{y_{m}}_{p} \big]\\
								&=\big(\norm{x-z}_{p}^{p}+c^{p}-\norm{z}_{p}^{p}\big)^{1/p} - c\\
								& =h^{\{z,c\}}(x),\; \text{ for all } x\in\ell_{p}(J). 
\end{align*}
Hence $h^{\{z,c\}}\in\overline{\ell_{p}(J)}^{h}$. 
Now, suppose that $h\in\overline{\ell_{p}(J)}^{h,\infty}\cup\{0\}$.
Then for every $x\in\ell_{p}(J)$ we have
\begin{equation*}
	h(x)=-\sum_{j\in J}\mu(j)x(j),
\end{equation*} 
for some $\mu\in\ell_{q}(J)$ with $\norm{\mu}_{q}\leq 1$.
Pick any countably infinite subset $K=\{j_{m}\}_{m=1}^{\infty}$ 
of $J$, and define the sequence $\{\mu_{m}\}_{m\in\N}$ in $\ell_{q}(J)$ 
by
\begin{equation}
	\mu_{m}(j):=\begin{cases}
			 \left( 1-\norm{\mu}_{q}^{q}+\abs{\mu(j)}^{q} \right)^{1/q}, & \text{ if } j=j_{m},\\
			 \mu(j), & \text{ if } j\neq j_{m}.
			 \end{cases}
\end{equation}
Then $\norm{\mu_{m}}_{q}=1$ for all $m\in\N$, and 
also $\mu_{m}\overset{w}\longrightarrow \mu$.
By $\ell_{p}/\ell_{q}$ duality, it follows that for each 
$m\in\N$ there exists $z_{m}\in\ell_{p}(J)$ with $\norm{z_{m}}_{p}=1$
such that $\sum_{j\in J}\mu_{m}(j)z_{m}(j)=1$. Therefore, by letting $y_{m}=mz_{m}$
and by proceeding as in the proof of \Cref{unblp}, we obtain
$h_{y_{m}}(x)\rightarrow h(x)$ for all $x\in\ell_{p}(J)$, as $m\to\infty$. 
Hence $ h$ belongs to $\overline{\ell_{p}(J)}^{h}$. Consequently,
the inclusion 
$\overline{\ell_{p}(J)}^{h,F}\sqcup\overline{\ell_{p}(J)}^{h,\infty}\subseteq \overline{\ell_{p}(J)}^{h}$
also holds.
\end{proof}
%*********************************************************************************************
\section*{Acknowledgement}
The author is very grateful to Prof. Kalle Kyt\"ol\"a, Prof. Anders Karlsson, 
and Prof. Olavi Nevanlinna for many valuable discussions and
suggestions. The author is also thankful to the anonymous referee 
for valuable suggestions that improved the presentation of 
this paper.

%%%%%%%%%%%%%%%%%%%%%%%%%%%%%%%%%%%%%%%%%%%%%%%%%%%%%%%%%%%%%%%%%%%%%%%%%%%%%%%%
\end{document}